\documentclass[11pt]{amsart}
\usepackage[T1]{fontenc}
\usepackage{lmodern,microtype}
\usepackage{amsmath,amssymb,amsthm,mathtools}
\usepackage[margin=1.5in]{geometry}
\usepackage{xcolor}
\usepackage{aliascnt}
\usepackage[colorlinks=true,linkcolor=blue!45!black,citecolor=blue!45!black,urlcolor=blue!45!black,pdfusetitle]{hyperref}
\usepackage[nameinlink,noabbrev]{cleveref}
\usepackage{enumitem}
\setlist{nosep}

\newtheorem{theorem}{Theorem}[section]

\newaliascnt{prop}{theorem}

\aliascntresetthe{prop}

\newaliascnt{cor}{theorem}

\aliascntresetthe{cor}

\newaliascnt{lemma}{theorem}
\newtheorem{lemma}[lemma]{Lemma}
\aliascntresetthe{lemma}

\renewcommand{\Re}{\operatorname{Re}}
\renewcommand{\Im}{\operatorname{Im}}

\title{The sharp constant in the Mashreghi--Ransford inequality}
\author[Ludovick Bouthat]{Ludovick Bouthat}
\email{ludovick.bouthat.1@ulaval.ca}
\address{D\'epartement de math\'ematiques et de statistique, Universit\'e Laval, Qu\'ebec, QC, Canada G1V 0A4.}



\begin{document}

\begin{abstract}
Let $(a_n)_{n\geq0}$ be a sequence of complex numbers, and define
\[
 b_n=\sum_{k=0}^n \Big(\!\!\!\begin{array}{c}
    n \\
    k 
 \end{array} \!\!\!\Big)a_k,
 \qquad
 c_n=\sum_{k=0}^n  \Big(\!\!\!\begin{array}{c}
    n \\
    k 
 \end{array} \!\!\!\Big)(-1)^{n-k}a_k.
\]
Let $\beta>1$, put $\alpha=\sqrt{\beta^2-1}$, and suppose that $b_n,c_n=O(\beta^n)$. Mashreghi and Ransford proved that
\[
  \limsup_{n\to\infty}\frac{|a_n|}{\alpha^n} \leq \kappa\! \left(\limsup_{n\to\infty}\frac{|b_n|}{\beta^n}\right)^{\!1/2}\! \left(\limsup_{n\to\infty}\frac{|c_n|}{\beta^n}\right)^{\!1/2}
\]
with a universal constant satisfying $2/\sqrt{3} \leq \kappa\leq2$.  We prove that the optimal constant is indeed $\kappa=2/\sqrt{3}$. The proof, which uses exponential generating functions, Phragm\'en--Lindel\"of estimates, and Cauchy's formula, is compared to the classical one of Mashreghi and Ransford.
\end{abstract}
\maketitle

\section{Introduction}

Given a complex sequence $(a_n)_{n\geq0}$, consider the two binomial transforms
\begin{equation}\label{eq:transforms}
 b_n=\sum_{k=0}^n\binom{n}{k}a_k, \qquad c_n=\sum_{k=0}^n\binom{n}{k}(-1)^{n-k}a_k.
\end{equation}
If $a_n=0$ for all $n>r$, where $r$ is a non-negative integer, then it is clear that $b_n,c_n = O(n^r)$. Surprisingly, Chalendar, Kellay and Ransford \cite{Chalendar} proved the converse: \emph{If $b_n,c_n = O(n^r)$, then $a_n=0$ for all $n>r$}. In the same paper, the authors extended this result by showing that if, for some $\beta>1$, $\max\{|b_n|,|c_n|\} \leq \beta^n (n+1)^r$, then there exists a positive constant $C(r,\beta)$ such that
$$
|a_n| \leq \begin{cases}
    C(0,\beta) \alpha^{n} \bigl(1+\log(n+1)\bigr) &\qquad \text{if } r=0;\\
    C(r,\beta) \alpha^{n-r} (n+1)^r &\qquad \text{if } r\geq 1,
\end{cases}
$$
where $\alpha:=\sqrt{\beta^2-1}$. Moreover, the authors conjectured that the factor $\log(n+1)$ could be removed by using a more precise argument. 

Mashreghi and Ransford \cite{MR} resolved this conjecture five years later by proving that, for $\beta>1$ and $\alpha=\sqrt{\beta^2-1}$,
\[
 b_n,c_n=O(\beta^n) \quad\Longrightarrow\quad a_n=O(\alpha^n).
\]
More precisely, they obtained the quantitative estimate
\[
  \limsup_{{n\to\infty}} \frac{|a_n|}{\alpha^n} \leq 2 \left(\limsup_{n\to\infty}\frac{|b_n|}{\beta^n}\right)^{\!1/2} \!\left(\limsup_{n\to\infty}\frac{|c_n|}{\beta^n}\right)^{\!1/2},
\]
and they showed by example that the universal constant cannot be smaller than $2/\sqrt3$. Indeed, it suffices to consider the sequence $a_n=(i\sqrt3)^n-(-i\sqrt3)^n$ with $\beta=2$ to obtain this conclusion.

Our main result closes this gap by showing that this smaller constant is in fact optimal.

\begin{theorem}\label{theorem:main}
Let $\beta>1$, define $\alpha=\sqrt{\beta^2-1}$, and suppose that the sequences in \eqref{eq:transforms} satisfy $b_n,c_n=O(\beta^n)$. Then
\begin{equation*}
 \limsup_{n\to\infty}\frac{|a_n|}{\alpha^n} \leq \frac{2}{\sqrt{3}} \left(\limsup_{n\to\infty}\frac{|b_n|}{\beta^n}\right)^{\!1/2} \!\left(\limsup_{n\to\infty}\frac{|c_n|}{\beta^n}\right)^{\!1/2},
\end{equation*}
and the factor $2/\sqrt{3}$ is optimal as a universal constant valid for all $\beta>1$.
\end{theorem}

The paper is organized as follows. In \Cref{sec - 2}, we reduce the proof to real sequences, normalize the two binomial transforms, and introduce the exponential generating functions used throughout the proof. In \Cref{sec - 3}, we revisit the proof of Mashreghi and Ransford, isolating the estimate responsible for the constant $2$ and identifying the point at which it can be improved. \Cref{sec - 4} is devoted to the main new ingredient: a sharper estimate for exponential generating functions, from which the factor $2/\sqrt{3}$ emerges. Finally, in \Cref{sec - 5}, a sectorial Phragm\'en--Lindel\"of argument allows us to complete the proof of \Cref{theorem:main}.

\section{Reduction to real sequences and normalization}\label{sec - 2}

We first observe that it suffices to prove \Cref{theorem:main} for real sequences. Indeed, for an admissible complex sequence, put
\[
 A_0:=\limsup_{n\to\infty}\frac{|a_n|}{\alpha^n}.
\]
If $A_0=0$, the conclusion is immediate. Otherwise, choose a subsequence $(n_j)$ such that $|a_{n_j}|/\alpha^{n_j}\to A_0$, allowing $A_0=\infty$. Passing to a further subsequence, we may choose $|\omega|=1$ so that $\omega a_{n_j}/|a_{n_j}|\to1$. The real sequence $a'_n:=\Re(\omega a_n)$ then satisfies
\[
 \limsup_{n\to\infty}\frac{|a'_n|}{\alpha^n}=A_0.
\]
Its binomial transforms are $b'_n=\Re(\omega b_n)$ and $c'_n=\Re(\omega c_n)$, whose normalized modulus limsups are no larger than those of $b_n$ and $c_n$. The result for real sequences therefore implies the complex case with the same constant. Henceforth we assume that $(a_n)$ is real.

Define
\[
 p:=\frac\alpha\beta \qquad\text{and}\qquad q:=\frac1\beta.
\]
Then $p,q>0$ and $p^2+q^2=1$.  Choose $0<\phi<\pi/2$ so that $p=\sin\phi$ and $q=\cos\phi$. Define
\[
 x_n:=\frac{a_n}{\alpha^n}, \qquad y_n:=\frac{b_n}{\beta^n}, \qquad z_n:= (-1)^n \frac{c_n}{\beta^n},
\]
and introduce the respective exponential generating functions
\begin{equation}\label{eq - def_func}
 F(w)=\sum_{n\geq0}x_n\frac{w^n}{n!},
 \qquad
 G(w)=\sum_{n\geq0}y_n\frac{w^n}{n!},
 \qquad
 H(w)=\sum_{n\geq0}z_n\frac{w^n}{n!}.
\end{equation}
By definition,
\[
 \limsup_{n\to\infty}|z_n| = \limsup_{n\to\infty}\frac{|c_n|}{\beta^n}.
\]
The hypothesis implies that $(y_n)$ and $(z_n)$ are bounded, so $G$ and $H$ are entire functions of exponential type at most one. Binomial inversion gives
\[
 a_n=\sum_{k=0}^n\binom nk(-1)^{n-k}b_k.
\]
Thus $|a_n|\le M(1+\beta)^n$ for some $M>0$, and consequently $F$ is entire and of finite exponential type. Moreover, the binomial theorem gives
\begin{equation}\label{eq:EGF}
 G(w)=e^{qw}F(pw) \qquad\text{and}\qquad H(w)=e^{qw}F(-pw).
\end{equation}
Eliminating $F$ in \eqref{eq:EGF} allows us to relate the two functions via the identities
\begin{equation}\label{eq:GH}
 e^{-2qw}G(w)=H(-w) \qquad\text{and}\qquad e^{-2qw}H(w)=G(-w).
\end{equation}
Since the coefficients are real, we further have
\begin{equation}\label{eq:real-symmetry}
 F(\overline w)=\overline{F(w)},\qquad
 G(\overline w)=\overline{G(w)},\qquad
 H(\overline w)=\overline{H(w)}.
\end{equation}
Moreover, since $(y_n)$ and $(z_n)$ are bounded, there exists $M>0$ such that
\begin{equation}\label{eq:GH-growth}
 |G(w)|\leq M e^{|w|}
 \qquad\text{and}\qquad
 |H(w)|\leq M e^{|w|}
 \qquad(w\in\mathbb C).
\end{equation}
These identities and bounds will be crucial to establishing decaying properties needed in \Cref{sec - 4}.

\section{The original argument}\label{sec - 3}

We recall the proof of \cite[Theorem~1.7]{MR} in a form that isolates the estimate to be improved. The first lemma is \cite[Lemma~5.1]{MR}; the second packages the coefficient-extraction step of their proof.

\begin{lemma}\label{lem:MR-PL}
Let $u$ be a subharmonic function on $\mathbb C$. Suppose that $\beta>1$ and $\gamma_+, \gamma_- \in\mathbb R$ satisfy
\[
u(z) \leq \beta|z|\pm\Re(z)+\gamma_{\pm}+o(1) \qquad (|z|\to\infty)
\]
Define $\alpha=\sqrt{\beta^2-1}$ and $T := \bigl\{z\in\mathbb C: |\Re z|<|z|^{3/4}\bigr\}.$ Then, as $|z|\to\infty$, 
\[
u(z) \leq \begin{cases}
    \alpha|\Im z|+\frac{\gamma_++\gamma_-}{2}+o(1) &\quad\text{if }z\in T,\\[3pt]
    \alpha\bigl(|z|-\tfrac{1}{2} |z|^{1/2}\bigr)+O(1) &\quad\text{if }z\notin T.
\end{cases}
\]
\end{lemma}

By hypothesis, $B := \limsup_{n\to\infty} |y_n| < \infty$. Hence, for every $\varepsilon>0$, choose $N$ such that $|y_n|\leq B+\varepsilon$ for $n\geq N$. Then, uniformly in $\arg z$,
\[
  |G(z)| \leq \sum_{n<N}|y_n|\frac{|z|^n}{n!} +(B+\varepsilon)\sum_{n\geq N}\frac{|z|^n}{n!} \leq o\bigl(e^{|z|}\bigr)+(B+\varepsilon)e^{|z|},
\]
and letting $\varepsilon\to 0$ yields $|G(z)|\leq (B+o(1))e^{|z|}$. Since $C := \limsup_{n\to\infty} |z_n|<\infty$, applying the same method also gives the bound $|H(z)|\leq (C+o(1))e^{|z|}$.

By \eqref{eq:EGF},
\[
 F(z)=e^{-z/\alpha}G(z/p) \qquad\text{and}\qquad F(z)=e^{z/\alpha}H(-z/p).
\]
Combining these identities with the preceding growth estimates while defining $u(z):=\alpha\log|F(z)|$ gives precisely the hypotheses of \Cref{lem:MR-PL} with $\gamma_-=\alpha\log B$ and $\gamma_+=\alpha\log C$, if $B,C>0$. If either constant is zero, apply the same argument with $B+\varepsilon$ and $C+\varepsilon$ and then let $\varepsilon\downarrow0$. Therefore, as $|z|\to\infty$,
\begin{equation}\label{eq:F-tube}
    |F(z)| \leq \begin{cases}
    \bigl(\sqrt{BC}+o(1)\bigr)e^{|\Im z|} &\quad\text{if }|\Re z| < |z|^{3/4},\\[3pt]
    e^{|z|-\frac{1}{2}|z|^{1/2}+O(1)} &\quad\text{if }|\Re z| \geq |z|^{3/4}.
    \end{cases}
\end{equation}

As a consequence of \eqref{eq:F-tube}, we have the following lemma, which is implicit in the proof of \cite[Theorem~1.7]{MR}. This result expresses the fact that, when estimating $|x_n|$ using Cauchy's formula, only the neighborhoods of the two saddle points $in$ and $-in$ contribute asymptotically. For a general account of saddle-point methods for generating functions, see \cite[Chapter~VIII]{FS}.

\begin{lemma}\label{lem:two-saddle-extraction}
Let $F$ be as defined in \eqref{eq - def_func} and $\eta_n := \arcsin\bigl(n^{-1/4}\bigr)$. Suppose that, for some $K\geq0$,
\begin{equation*}
  \sup_{|\theta|\leq\eta_n} e^{-n\cos\theta} \bigl( |F(-ine^{i\theta})| + |F(ine^{-i\theta})|\bigr) \leq K+o(1).
\end{equation*}
Then
\[
\limsup_{n\to\infty}|x_n| \leq K.
\]
\end{lemma}
\begin{proof}
Let $T_n:= \bigl\{ t\in[-\pi,\pi]: ne^{it} \in T \bigr\}$, where $T$ is defined in \Cref{lem:MR-PL}. By Cauchy's formula,
\begin{align*}
    |x_n| &= \biggl|\frac{n!}{2\pi n^n} \int_{-\pi}^{\pi} F(ne^{it})e^{-int} \,\mathrm{d}t \biggr| \\
    &\leq\, \frac{n!}{2\pi n^n} \int_{[-\pi,\pi]\setminus T_n}|F(ne^{it})| \,\mathrm{d}t +  \frac{n!}{2\pi n^n} \int_{T_n}|F(ne^{it})| \,\mathrm{d}t\\
    &= \frac{n!}{2\pi n^n} \int_{[-\pi,\pi]\setminus T_n}|F(ne^{it})| \,\mathrm{d}t + \frac{n!}{2\pi n^n} \int_{-\eta_n}^{\eta_n} \left( |F(-ine^{it})| + |F(ine^{-it})|\right) \mathrm{d}t.
\end{align*}
We claim that the first term vanishes as $n\to\infty$. Indeed, on $[-\pi,\pi]\setminus T_n$, \eqref{eq:F-tube} ensures that
\begin{align*}
    \int_{[-\pi,\pi]\setminus T_n} |F(ne^{it})|\,\mathrm{d}t &\leq \int_{[-\pi,\pi]\setminus T_n} e^{n-\frac{1}{2}\sqrt{n}+O(1)} \,\mathrm{d}t \leq 2\pi e^{n-\frac{1}{2}\sqrt{n}+O(1)}.
\end{align*}
Hence, it follows from Stirling's formula that
\[
  \frac{n!}{2\pi n^n} \int_{[-\pi,\pi]\setminus T_n} |F(ne^{it})|\,\mathrm{d}t \leq \frac{n!e^{n-\frac{1}{2}\sqrt{n}+O(1)}}{n^n} \asymp \sqrt{\frac{2\pi n}{e^{\sqrt{n}}}} \longrightarrow 0.
\]
Consequently, the hypothesis implies that
\[
  |x_n| \leq (K+o(1)) \frac{n!}{2\pi n^n} \int_{-\pi}^{\pi}e^{n\cos t}\,\mathrm{d}t +o(1).
\]
Stirling's formula and elementary computations then give
\[
  \frac{n!}{n^n} \cdot \frac{1}{2\pi} \int_{-\pi}^{\pi}e^{n\cos\theta}\,\mathrm{d}\theta \sim \frac{\sqrt{2\pi n} n^n e^{-n}}{n^n} \cdot \frac{e^n}{\sqrt{2\pi n}} = 1,
\]
and the conclusion follows.
\end{proof}

This lemma allows us to easily complete the original proof of Mashreghi--Ransford. Using \eqref{eq:F-tube}, we find that
\[
  |F(\pm ine^{\mp i\theta})| \leq \bigl(\sqrt{BC}+o(1)\bigr)e^{n\cos\theta}, \qquad (n\to\infty),
\]
uniformly for $|\theta|\leq\eta_n$. Adding these two inequalities gives, as $n\to\infty$,
\begin{equation}\label{eq:old-paired-bound}
  |F(-ine^{i\theta})|+|F(ine^{-i\theta})| \leq \bigl(2\sqrt{BC}+o(1)\bigr)e^{n\cos\theta}.
\end{equation}
\Cref{lem:two-saddle-extraction}, with $K=2\sqrt{BC}$, therefore yields
\[
\limsup_{n\to\infty}|x_n| \leq 2\sqrt{BC}.
\]
Since $x_n=a_n/\alpha^n$, this is precisely the estimate obtained in \cite[Theorem 1.7]{MR}.

\medskip
The factor $2$ in \eqref{eq:old-paired-bound} comes from the direct estimation in \Cref{eq:F-tube}. We shall instead combine \Cref{lem:edge} with \Cref{lem:PL} to prove
\begin{equation}\label{eq:new-paired-bound}
  |F(-ine^{i\theta})|+|F(ine^{-i\theta})| \leq \left(\frac{2}{\sqrt3}\sqrt{BC}+o(1)\right)e^{n\cos\theta}
\end{equation}
uniformly for $|\theta|\leq\eta_n$. Applying the unchanged \Cref{lem:two-saddle-extraction}, now with $K=\frac{2}{\sqrt3}\sqrt{BC}$, will give the sharp inequality
\[
\limsup_{n\to\infty}\frac{|a_n|}{\alpha^n} \leq \frac{2}{\sqrt3}\sqrt{BC}.
\]
Hence, the main goal for the remainder of this paper is to prove the improved bound \eqref{eq:new-paired-bound}.

\section{An improved estimate}\label{sec - 4}

This section is dedicated to proving \Cref{lem:edge}, which is used to obtain the improved estimate
\[
  |F(-ine^{i\theta})|+|F(ine^{-i\theta})| \leq \left( \frac{2}{\sqrt3}\sqrt{BC}+o(1) \right)e^{n\cos\theta}.
\]
To do so, we begin with a general lemma on the decay rate of an arbitrary exponential generating function with real coefficients along the ray of argument $\phi$ under a decay assumption on the special ray of argument $3\phi$. 

\begin{lemma}\label{lem:edge}
Let $0<\phi<\pi/2$, let $(s_n)_{n\geq0}$ be a bounded real sequence, and define $E(w)=\sum_{n\geq0}s_n\frac{w^n}{n!}.$ Suppose that
\begin{equation}\label{eq:edge-decay}
 e^{-t}E(te^{3i\phi})\longrightarrow0 \qquad (t\to\infty).
\end{equation}
Then
\[
 \limsup_{t\to\infty}e^{-t} |E(te^{i\phi})| \leq \frac{1}{\sqrt3}\limsup_{n\to\infty}|s_n|.
\]
\end{lemma}
\begin{proof}
Let $L:=\limsup_{n\to\infty}|s_n|$. For every $\varepsilon>0$, setting finitely many coefficients equal to zero gives a sequence bounded in modulus by $L+\varepsilon$. This changes $E$ only by a polynomial, preserving \eqref{eq:edge-decay} and leaving the normalized limsup in the conclusion unchanged. Dividing by $L+\varepsilon$ and then letting $\varepsilon\downarrow0$, it therefore suffices to prove the bound $1/\sqrt3$ under the assumption $|s_n|\leq1$. Now, for each $k\geq1$, define
\[
w_n=e^{-t}\frac{t^n}{n!}, \qquad h_{n,k}:=\cos\bigl(k(n\phi+\theta)\bigr),
\]
and observe that we have
\begin{equation}\label{eq:hk-E}
  \sum_{n\geq 0} w_n s_n h_{n,k} = e^{-t}\Re\!\left( e^{ik\theta}E(te^{ik\phi}) \right).
\end{equation}
For $k=1$, we may choose $\theta$ (possibly depending on $\phi$ and $t$) so that
\[
  \Bigg| \sum_{n\geq 0} w_n s_n h_{n,1} \Bigg| = e^{-t} |E(te^{i\phi})|.
\]
Hence, we seek to bound the left-hand side uniformly in $\theta$. To do so, we relate $e^{-t} E(te^{i\phi})$ to $e^{-t} E(te^{3i\phi})$ via the trigonometric identity $4h_{n,1}^3=3h_{n,1}+h_{n,3}$ to obtain
\begin{align*}
    3e^{-t} |E(te^{i\phi})| &= 3\,\Bigg| \sum_{n\geq 0} w_n s_n h_{n,1} \Bigg| =  \Bigg|4\sum_{n\geq 0} w_n s_n h_{n,1}^3 - \sum_{n\geq 0} w_n s_n h_{n,3}\Bigg| \\
    &= \Bigg| 4\sum_{n\geq 0} w_n s_n h_{n,1}^3 - e^{-t}\Re\!\left( e^{3i\theta}E(te^{3i\phi})\right)\Bigg| \\
    &= 4 \,\Bigg|\sum_{n\geq 0} w_n s_n h_{n,1}^3\Bigg| + o(1),
\end{align*}
where we used \eqref{eq:edge-decay} and \eqref{eq:hk-E}. Since $|s_n|\leq 1$, Cauchy--Schwarz yields
\begin{equation}\label{eq - CS}
\begin{aligned}
    3e^{-t} |E(te^{i\phi})|  &\leq 4\Bigg( \sum_{n\geq 0} w_n |s_n|^2 h_{n,1}^2 \Bigg)^{\!\!1/2}\! \Bigg( \sum_{n\geq 0} w_n h_{n,1}^4 \Bigg)^{\!\!1/2}\! +o(1) \\
    &\leq 4\Bigg( \sum_{n\geq 0} w_n h_{n,1}^2 \Bigg)^{\!\!1/2}\! \Bigg( \sum_{n\geq 0} w_n h_{n,1}^4 \Bigg)^{\!\!1/2}\! +o(1).
\end{aligned}
\end{equation}
Hence, we only need to provide sharp upper bounds for the remaining two terms, which are independent of $s_n$. Observe that
\[
  \sum_{n\geq 0} w_n e^{ikn\phi} = e^{-t}\sum_{n\geq 0}\frac{(te^{ik\phi})^n}{n!} = e^{t(e^{ik\phi}-1)}.
\]
Consequently, for $k=2$ and $k=4$, we have
\[
  \sum_{n\geq 0} w_n h_{n,k} = \Re\!\left( e^{ik\theta}e^{t(e^{ik\phi}-1)} \right) = o(1),
\]
uniformly in $\theta$, since
\(
  \big|e^{t(e^{ik\phi}-1)}\big| = e^{t(\cos(k\phi)-1)} \longrightarrow 0.
\)
The trigonometric identities $2h_{n,1}^2=1+h_{n,2}$ and $8h_{n,1}^4=3+4h_{n,2}+h_{n,4}$ thus imply that
\begin{equation}\label{eq:moments-h}
  \sum_{n\geq 0} w_n h_{n,1}^2=\frac{1}{2}+o(1), \qquad \sum_{n\geq 0} w_n h_{n,1}^4=\frac{3}{8}+o(1),
\end{equation}
uniformly in $\theta$. Therefore, using \eqref{eq - CS}, we finally find
\begin{align*}
    3e^{-t} |E(te^{i\phi})| &\leq 4\left(\frac{1}{2}+o(1)\right)^{\!1/2}\!\left(\frac{3}{8}+o(1)\right)^{\!1/2}+o(1) \\
    &= \sqrt{3}+o(1),
\end{align*}
uniformly in $\theta$. Dividing both sides by $3$ yields the desired result.
\end{proof}

We now verify that the generating functions $G$ and $H$ associated to the sequences $(b_n)$ and $(c_n)$ satisfy the decay assumption in \Cref{lem:edge}. This is achieved with the help of the following lemma, which relies almost entirely on a sectorial Phragm\'en--Lindel\"of argument. We recall that the classical sectorial Phragm\'en--Lindel\"of theorem says that if a function $f$ is holomorphic in $S_\alpha=\{|\arg z|<\pi/(2\alpha)\}$, continuous on its closure, bounded by $M$ on the boundary, and satisfies $|f(z)|=O(e^{|z|^\beta})$ for some $\beta<\alpha$, then $|f|\le M$ throughout $S_\alpha$  (see \cite[Theorem 1.4.2]{Boas}). 

We shall also use the following asymptotic form in Section~\ref{sec - 5}. If the boundary condition is imposed only outside a compact set and $\limsup_{r\to\infty}
|f(re^{\pm i\pi/2\alpha})|\le M,$ then
\[
\limsup_{r\to\infty}|f(re^{i\theta})|\le M,
\]
uniformly when \(\theta\) ranges over a compact subinterval of $\left(-\frac{\pi}{2\alpha},\frac{\pi}{2\alpha}\right).$ This exterior-sector version follows from the usual principle by applying it on a truncated sector to $\frac{w}{w+\lambda}f(w),$ with \(\lambda>0\) chosen to control the circular part of the boundary, and then letting \(|w|\to\infty\); see \cite[Section~5.63]{Titchmarsh}.

\begin{lemma}\label{lem:third-harmonic-decay}
Let $0<\phi<\pi/2$ and $q=\cos\phi$. Suppose that $E$ is entire and, for some $M>0$,
\[
 |E(w)|\leq M e^{|w|}, \qquad |e^{-2qw}E(w)|\leq M e^{|w|} \qquad(w\in\mathbb C).
\]
Then
\[
 E(te^{\pm3i\phi})=o(e^t) \qquad (t\to\infty).
\]
In particular, $E$ satisfies \eqref{eq:edge-decay}.
\end{lemma}

\begin{proof}
Set $\widetilde E(w):=e^{-2qw}E(w)$ and consider the entire functions
\[
 J(w):=e^{-e^{-i\phi}w}E(w)
\]
in the sector $\phi \leq \arg w \leq \pi-\phi$. On the first boundary ray, the growth bound for $E$ gives
\[
 |J(te^{i\phi})|=e^{-t}|E(te^{i\phi})|\leq M.
\]
Since $-e^{-i\phi}=e^{i\phi}-2q$, we also have $J(w)=e^{e^{i\phi}w}\widetilde E(w)$, and hence on the other boundary ray
\[
 |J(te^{i(\pi-\phi)})| = e^{-t}|\widetilde E(te^{i(\pi-\phi)})| \leq M.
\]
The function $J$ has finite exponential type, and the sector has opening $\pi-2\phi<\pi$. The sectorial Phragm\'en--Lindel\"of principle therefore gives $|J(w)|\leq M$ throughout the sector. Consequently,
\begin{equation}\label{eq:sector-growth}
 |E(te^{i\theta})| \leq M e^{t\cos(\theta-\phi)} \qquad(\phi\leq\theta\leq\pi-\phi).
\end{equation}
The same argument in the reflected sector, applied to $e^{-e^{i\phi}w}E(w)$, gives
\[
 |E(te^{-i\theta})| \leq M e^{t\cos(\theta-\phi)} \qquad(\phi\leq\theta\leq\pi-\phi).
\]
If $0<\phi\leq\pi/4$, then $3\phi\in[\phi,\pi-\phi]$, so these two estimates yield
\[
 |E(te^{\pm3i\phi})|\leq M e^{t\cos(2\phi)}.
\]
If $\pi/4<\phi<\pi/2$, then $\cos(3\phi)<0$, and the growth bound for $\widetilde E$ gives directly
\[
 |E(te^{\pm3i\phi})| = e^{2qt\cos(3\phi)}|\widetilde E(te^{\pm3i\phi})| \leq M e^{t(1+2q\cos(3\phi))}.
\]
In both cases the coefficient of $t$ in the exponent is strictly less than $1$, which proves the required decay.
\end{proof}

By \eqref{eq:GH} and \eqref{eq:GH-growth}, \Cref{lem:third-harmonic-decay} applies to both $E(w)=G(w)$ and $E(w)=H(w)$, with respective partners $H(-w)$ and $G(-w)$. Their coefficient sequences $(y_n)$ and $(z_n)$ are real and the limsups of their moduli are $B$ and $C$. Applying \Cref{lem:edge} and using \eqref{eq:real-symmetry} therefore gives
\begin{equation}\label{eq:Gedge}
\begin{aligned}
 \limsup_{t\to\infty}e^{-t}|G(te^{\pm i\phi})| &\leq \frac{B}{\sqrt3},\\
 \limsup_{t\to\infty}e^{-t}|H(te^{\pm i\phi})| &\leq \frac{C}{\sqrt3}.
\end{aligned}
\end{equation}

\section{A Phragm\'en--Lindel\"of argument and proof of \texorpdfstring{\Cref{theorem:main}}{Theorem 1.1}}\label{sec - 5}

If $BC=0$, the estimate $\limsup_{n\to\infty}|x_n|\leq2\sqrt{BC}$ proved in \Cref{sec - 3} already gives the result. Henceforth assume $B,C>0$. Let $\delta:=\pi/2-\phi$ and define
\begin{equation}\label{eq:Psi}
 \Psi(w):=e^{-pw}F(-ipw).
\end{equation}
This is an entire function of finite exponential type. Since $e^{i\delta}=p+iq$, the identities in \eqref{eq:EGF} give, for $t>0$,
\begin{equation}\label{eq:Psi-boundary}
\begin{aligned}
 \Psi(te^{i\delta})&=e^{-t}G(te^{-i\phi}),\\
 \Psi(te^{-i\delta})&=e^{-t}H(te^{i\phi}).
\end{aligned}
\end{equation}

We then use the following two-constants form of the Phragm\'en--Lindel\"of principle.

\begin{lemma}\label{lem:PL}
Let $0<\delta<\pi/2$, and let $f$ be holomorphic in $|\arg w|<\delta$, continuous on the boundary outside a compact set, and of exponential growth in the sector. Suppose that
\[
 \limsup_{r\to\infty}|f(re^{i\delta})|\leq M_+,
 \qquad
 \limsup_{r\to\infty}|f(re^{-i\delta})|\leq M_-,
\]
where $M_\pm>0$. Then, for $-\delta<\theta<\delta$,
\[
 \limsup_{r\to\infty}|f(re^{i\theta})|  \leq  M_+^{(\delta+\theta)/(2\delta)} M_-^{(\delta-\theta)/(2\delta)}.
\]
The estimate is uniform when $\theta$ ranges over a compact subinterval of $(-\delta,\delta)$.
\end{lemma}

\begin{proof}
Fix $\varepsilon>0$. Using the branch of $\operatorname{Log}$ in the sector, define the function
\[
  g(w) := \frac{e^{-\varepsilon}}{\sqrt{M_+ M_-}} \exp\!\left(  i\,\frac{\log(M_+/M_-)}{2\delta}\operatorname{Log} w \right)f(w).
\]
The factor multiplying $f$ is chosen so that its modulus interpolates between $M_+^{-1}$ and $M_-^{-1}$ on the two boundary rays. Indeed, for $w=re^{i\theta}$,
\[
  |g(re^{i\theta})| = e^{-\varepsilon} M_+^{-(\delta+\theta)/(2\delta)} M_-^{-(\delta-\theta)/(2\delta)} |f(re^{i\theta})|.
\]
Hence
\[
  \limsup_{r\to\infty}|g(re^{\pm i\delta})| \le e^{-\varepsilon}<1
\]
so there exists some $R>0$ such that
\[
  |g(re^{\pm i\delta})|\le1,\qquad r\ge R,
\]
Moreover, \(g\) is holomorphic in the sector \(|\arg w|<\delta\), continuous on its boundary outside a compact set, and of exponential growth, since the modulus of the additional factor is independent of \(|w|\). As \(2\delta<\pi\), the asymptotic form of the sectorial Phragmén–Lindelöf principle therefore gives
\[
\limsup_{r\to\infty}|g(re^{i\theta})|\le1,
\qquad |\theta|<\delta,
\]
uniformly on compact subintervals. Thus
\[
  \limsup_{r\to\infty}|f(re^{i\theta})| \le e^\varepsilon M_+^{(\delta+\theta)/(2\delta)} M_-^{(\delta-\theta)/(2\delta)}.
\]
Letting $\varepsilon\downarrow0$ proves the result.
\end{proof}

Apply \Cref{lem:PL} to $\Psi$. By \eqref{eq:Gedge} and \eqref{eq:Psi-boundary}, we may take $M_+=B/\sqrt3$ and $M_-=C/\sqrt3$. Thus
\[
 \limsup_{r\to\infty}|\Psi(re^{i\theta})| \leq \frac1{\sqrt3} B^{(\delta+\theta)/(2\delta)}C^{(\delta-\theta)/(2\delta)},
\]
uniformly for $\theta$ in compact subintervals of $(-\delta,\delta)$. Since $\eta_n=\arcsin(n^{-1/4})\to0$, setting $r=n/p$ and using \eqref{eq:Psi} gives
\[
 \sup_{|\theta| \leq \eta_n} e^{-n\cos\theta}|F(-ine^{i\theta})| \leq \frac1{\sqrt3}\sqrt{BC}+o(1).
\]
By \eqref{eq:real-symmetry}, the two saddle contributions have equal modulus:
\[
 |F(ine^{-i\theta})|=|F(-ine^{i\theta})|.
\]
Therefore,
\begin{equation}\label{eq:middle-discrete}
 \sup_{|\theta|\leq\eta_n} e^{-n\cos\theta} \bigl(|F(-ine^{i\theta})|+|F(ine^{-i\theta})|\bigr) \leq \frac2{\sqrt3}\sqrt{BC}+o(1).
\end{equation}
Applying \Cref{lem:two-saddle-extraction} with $K=2\sqrt{BC}/\sqrt3$ now yields
\[
 \limsup_{n\to\infty}|x_n|\leq\frac2{\sqrt3}\sqrt{BC}.
\]
Since $x_n=a_n/\alpha^n$, this proves the inequality for real sequences. The reduction in \Cref{sec - 2} gives the complex case, and the example in the introduction proves optimality, completing the proof of
\Cref{theorem:main}.

\medskip
\noindent{\bf AI disclosure statement.}
ChatGPT 5.6 Sol by OpenAI was used to explore proof strategies for this paper and help with its writing.  All mathematical arguments were independently verified by the author, who takes full responsibility for the content.

\end{document}